\documentclass{amsart}

\usepackage{color,latexsym,amsmath,amssymb,amsfonts,amsthm}

\newtheorem{lemma}{Lemma}[section]
\newtheorem{theorem}[lemma]{Theorem}
\newtheorem{proposition}[lemma]{Proposition}
\newtheorem{corollary}[lemma]{Corollary}

\theoremstyle{definition}
\newtheorem{definition}[lemma]{Definition}

\theoremstyle{remark}
\newtheorem{example}[lemma]{Example}

\newtheorem{remark}[lemma]{Remark}

\def\SQ{\mathbb Q}    
\def\SZ{\mathbb Z}    
\def\SR{\mathbb R}    

\def\mm{\mathfrak m}    
\def\pp{\mathfrak p}    
\def\aa{\mathfrak a}    
\def\bb{\mathfrak b}    

\def\II{\mathfrak I}    
\def\MM{\mathfrak M}    
\def\PP{\mathfrak P}    

\def\int{\mbox{\rm{Int}}}             
\def\intz{\int (\SZ)}                 

\def\hV{\widehat{V}}            
\def\hm{\widehat{\mm}}          
\def\hE{\widehat{E}}            

\def\g{\mathcal g}
\def\Mm{\mathcal M}         
\def\Oo{\mathcal O}         
\def\Aa{\mathcal A}   
\def\Bb{\mathcal B}         

\def\Ž{\'e}
\def\{\`e}
\def\ˆ{\`a}
\def\¬ù{\`u}
\def\¬ô{\^{o}}
\def\{\^{e}}
\def\¬û{\^{u}}
\def\{\c c}

\newcommand\NC {\normalcolor}

\def\Uu{\mathcal U}

\def\cont{\mathrm{cont}}

\def\be{\begin{equation}}
\def\ee{\end{equation}}

\title[Integer-valued polynomials and the stacked bases property]{Integer-valued polynomials, Pr\"ufer domains\\ and the stacked bases property}

\author{Jacques Boulanger}

\author{Jean-Luc Chabert*}
\address{LAMFA CNRS-UMR 7352\\
Universit\'e de Picardie\\
80039 Amiens, France }

\email{jean-luc.chabert@u-picardie.fr}


\begin{document}

\maketitle

\begin{abstract}
To study the question of whether every two-dimensional Pr\"ufer domain possesses the stacked bases property, we consider the particular case of the Pr\"ufer domains formed by integer-valued polynomials.  The description of the spectrum of the rings of integer-valued polynomials on a subset of a rank-one valuation domain enables us to prove that they all possess the stacked bases property. We also consider integer-valued polynomials on rings of integers of number fields and we reduce in this case the study of the stacked bases property to questions concerning $2\times 2$-matrices.
\end{abstract}


\section{Introduction}

The stacked bases property -or simultaneous bases property- was introduced in the attempts to generalize to domains the stacked bases theorem of finitely generated abelian groups or, more generally, of finitely generated modules over principal ideal domains.   

\begin{definition}(see~\cite[Chapter V, \S 4]{bib:FS})
An integral domain $D$ is said to have {\em the stacked bases property} if, for every free $D$-module $M$ with finite rank $m$ and every finitely generated submodule $N$ of $M$ of rank $n\leq m$, there exist rank-one projective $D$-modules $P_1, P_2, \cdots, P_m$ and nonzero ideals $\II_1, \II_2, \cdots, \II_n$ of $D$ such that:
$$M=P_1\oplus P_2\oplus\cdots\oplus P_m\quad,\quad  N=\II_1P_1\oplus \II_2P_2 \oplus\cdots\oplus \II_nP_n \,,$$ 
$$\textrm{and }\;\;\II_{j+1}\subseteq \II_j \;\textrm{ for } \;1\leq j\leq n-1\,. $$
\end{definition}

It is known that every Dedekind domain has the stacked bases property~\cite[Theorem 22.12]{bib:CR}. If we try to extend such a result to non-Noetherian domains, we are led to consider Pr\"ufer domains and it is known~\cite{bib:BK1} that every one-dimensional Pr\"ufer domain has the stacked bases property as well as every Pr\"ufer domain of finite character (that is, such that every nonzero element is contained in at most finitely many maximal ideals). Thus, since 1987, the question raised by Brewer, Katz and Ullery~\cite{bib:BKU} was: 

{\em Does every Pr\"ufer domain have the stacked bases property?}

\smallskip

For long ago, we also knew~\cite[4.3 and 4.4]{bib:chabert1972} that the ring of integer-valued polynomials 
$$\int(\SZ)=\{f(X)\in\SQ[X]\mid f(\SZ)\subseteq\SZ\}$$ 
is a two-dimensional Pr\"ufer domain that is not of finite character. 
That is why Heinzer, who was well aware of this fact, suggested to Brewer to ``try integer-valued polynomials''~\cite{bib:brewer2000}. This very natural ring $\int(\SZ)$ could be a counterexample as it was for several questions of commutative algebra. In fact, Brewer was not able to answer the question for the ring $\int(\SZ)$ but, with Klingler, he obtained nice results by considering other rings of integer-valued polynomials.

Recall that, for any domain $D$ with quotient field $K$ and any subset $E$ of $D$, the ring of {\em integer-valeud polynomials on $E$ with respect to $D$} is:
$$\int(E,D)=\{f(X)\in K[X]\mid f(E)\subseteq D\}\,.$$
Brewer and Klingler first proved in 1991 that:

\begin{proposition}\cite{bib:BK3}\label{th:12C}
If $D$ is a semi-local principal domain with finite residue fields then the ring $\int(D)=\int(D,D)$ is a two-dimensional Pr\"ufer domain which has the stacked bases property.
\end{proposition}

More interesting is the following:

\begin{theorem}\label{th:12A}
Let $V$ be an $n$-dimensional valuation domain and let $E$ be a subset of $V$ which is assumed to be precompact (that is, its completion is compact). Then, 
\begin{enumerate}
\item The domain $\int(E,V)$ is Pr\"ufer with dimension $n+1$ \cite{bib:CCL}.
\item The domain $\int(E,V)$ has the stacked bases property 	\cite{bib:BK2}.
\end{enumerate}
\end{theorem}
 
Consequently, there are Pr\"ufer domains of any dimension which are not of finite character and which have the stacked bases property. The remaining question is then: 

\smallskip

\noindent{\bf Question}.
{\em Are there two-dimensional Pr\"ufer domains which do not have the stacked bases property? In particular, does $\int(\SZ)$ have the stacked bases property?}

\smallskip

We may notice that in the examples studied by Brewer and Klinger, the ground domain $V$ is quasi-local (Theorem~\ref{th:12A}) or semi-local (Proposition~\ref{th:12C}), while $\SZ$ is not semi-local. 
In this paper, we give no answer, just contributions by focusing on rings of integer-valued polynomials which are two-dimensional Pr\"ufer domains. There are two kinds of results: in a first part, concerning the `local case' (sections 2, 3 and 4), we show that, for every rank-one valuation domain $V$ and every subset $E$ of $V$, if $\int(E,V)$ is a Pr\"ufer domain, then $\int(E,V)$ has the stacked bases property (Theorem~\ref{th:43D}). In order to prove this result, using Frisch's results~\cite{bib:frisch2017}, we give first a complete description of the spectrum of $\int(E,V)$ when $\int(E,V)$ is Pr\"ufer (Theorem~\ref{th:35D}). In a second part (section 5), we study the `global case' with $\int(\Oo_K)$ where $\Oo_K$ denotes the ring of integers of a number field $K$ and we reduce the question to know whether $\int(\Oo_K)$ has the stacked bases property to questions concerning $2\times 2$-matrices (Theorem~\ref{th:56D}) by continuing the work undertaken in~\cite{bib:chabert2012}.


\section{Local study: Pr\"ufer domains and pseudo-monotone sequences}

\noindent{\bf Notation}.
{\em Let $K$ be a valued field, that is, a field endowed with a rank-one valuation $v$} [$v(K^*)$ is a subgroup of $(\SR,+)$\,]. {\em We denote by $V$ the valuation domain, $\mm$ its maximal ideal, and $E$ a subset of $V$.}

\medskip


In view of the first assertion of Theorem~\ref{th:12A}, the first question that arises and which is posed in \cite{bib:CCL} is:  

\smallskip

{\em The precompactness of $E$ is sufficient for $\int(E,V)$ to be a Pr\"ufer domain. Is it necessary?}

\smallskip

Here are partial answers given to this question:

\noindent $\bullet$ Yes, when $V$ is a discrete valuation domain \cite{bib:CCL}.

\noindent $\bullet$ Yes, when $E$ is a subgroup of the group $(V,+)$ \cite{bib:park}.

\noindent $\bullet$ Yes, when $E$ is a regular subset of $V$ \cite{bib:chabertgraz2016} (generalized regular subsets in Amice's sense are defined in \cite{bib:CEF}; every additive or multiplicative subgroup of $V$ is regular).

\noindent $\bullet$ Yes, when the completion $\widehat{V}$ of $V$ is maximally complete \cite{bib:chabertgraz2016} (that is, when $\widehat{V}$ possesses no proper immediate extensions; this is the case when $V$ is discrete).

\noindent $\bullet$ No in general: if $E$ is formed by the elements of a pseudo-convergent sequence of transcendental type, then $\int(E,V)$ is Pr\"ufer while $E$ is a non-precompact subset \cite{bib:LW} (the definition of pseudo-convergence is recalled below).

\smallskip

Finally, this last counterexample suggested to Peruginelli the following characterization:
 
\begin{theorem}\cite{bib:giulio2018}\label{th:21A}
The ring $\int(E,V)$ is Pr\"ufer if and only if the only pseudo-monotone sequences contained in $E$ are either Cauchy sequences or pseudo-conver\-gent sequences of transcendental type.
\end{theorem}

Let us recall now what are pseudo-monotone sequences. The notion of pseudo-convergent sequence was introduced by Ostrowski \cite[\S 11,\, n$^0$ 62]{bib:ostrowski} and later extended by symmetry in \cite{bib:chabert2010} to characterize the polynomial closure $\overline{E}$ of any subset $E$ of $V$.

\begin{definition}
A sequence $\{x_n\}_{n\geq 0}$ of elements of $V$ is said to be {\em pseudo-monotone} if one of the three following conditions hold:
\begin{enumerate}
\item the sequence is {\em pseudo-convergent}

\centerline{$\forall\, n>m>l\geq 0 \quad \quad v(x_n-x_m) > v(x_m-x_l)\,,$}

\item the sequence is {\em pseudo-divergent}

\centerline{$\forall\, n>m>l\geq 0 \quad \quad v(x_n-x_m) < v(x_m-x_l)\,,$}

\item the sequence is {\em pseudo-stationary}

\centerline{$\forall\, n>m>l\geq 0 \quad \quad v(x_n-x_m)=v(x_m-x_l)\,.$}
\end{enumerate}
\end{definition}

Moreover, a pseudo-convergent sequence $\{x_n\}_{n\geq 0}$ is said to be of {\em transcendental type} if, for every polynomial $f(X)\in K[X]$, the sequence $\{v(f(x_n))\}_{n\geq 0}$ is eventually stationary.

\smallskip

In order to prove that, if the ring $\int(E,V)$ is Pr\"ufer, then it has the stacked bases property, we need to describe the spectrum of $\int(E,V)$. In fact, we do not really need to know the characterization given by Theorem~\ref{th:21A} since to obtain this description of the spectrum, we only need to know the following:

\begin{lemma}\cite{bib:chabertgraz2016}\label{th:23D}
If the ring $\int(E,V)$ is Pr\"ufer, then $E$ does not contain any sequence which is either pseudo-divergent or pseudo-stationary.
\end{lemma}

\begin{proof}
Assume that the conclusion of Lemma~\ref{th:23D} does not hold: $E$ contains a sequence $\{x_n\}_{n\geq 0}$ which is either pseudo-divergent or pseudo-stationary. Then, it follows from \cite[Prop. 4.8]{bib:chabert2010} that the closed ball $$B(x_0,v(x_1-x_0))=\{y\in V\mid v(y-x_0)\geq v(x_1-x_0)\}$$ is contained in the polynomial closure $\overline{E}$ of $E$, where $$\overline{E}=\{y\in V\mid \forall f\in\int(E,V)\; \;f(y)\in V\}.$$
Consequently, 
$$\int(E,V)=\int(\overline{E},V)\subseteq \int(B(x_0,v(x_1-x_0)),V)\,. $$
Now, note that the existence in $V$ a pseudo-divergent sequence implies that $v$ is not discrete and the existence of a pseudo-stationary sequence implies that the residue field $V/\mm$ is infinite. It is known that in both cases \cite[I.3.16]{bib:CC}:
$$\int(B(0,1),V)=\int(V)=V[X]\,.$$ 
Thus
$$\int(B(x_0,v(x_1-x_0)),V) = V\left[\frac{X-x_0}{x_1-x_0}\right]$$
and $\int(E,V)$ cannot be Pr\"ufer since its overring $V\left[\frac{X-x_0}{x_1-x_0}\right]$ is not.
\end{proof}

\begin{remark}
Let $W$ be any valuation domain and let $E$ be a subset of $W$ such that $\int(E,W)$ is a Pr\"ufer domain. We know that it is the case when $E$ is precompact~\cite[Theorem 4.1]{bib:CCL}. If $E$ is not precompact, following the beginning of the proof of~\cite[Theorem 2.7]{bib:park}, $W$ has a height-one prime ideal $\pp$ and the domain $\int(E,W_{\pp})$ is Pr\"ufer. Consequently, $E$ has to satisfy 	the conditions of Theorem~\ref{th:21A} with respect to the rank-one valuation domain $V=W_{\pp}$.\end{remark}

\section{The spectrum of $\int(E,V)$}

We look first at what can be said when there is no pseudo-divergent sequence.

\begin{lemma}\label{th:31D}
If $E$ does not contain any pseudo-divergent sequence, then $E$ admits $v$-orderings.
\end{lemma}

Let us recall the notion of $v$-ordering introduced by Bhargava~\cite{bib:bhargava1}: a {\em $v$-ordering of $E$} is a sequence $\{a_n\}_{n\geq 0}$ of elements of $E$ where, for every $n\geq 1$, $a_n$ satisfies $$v\left(\prod_{k=0}^{n-1}(a_n-a_k)\right)=\inf_{x\in E}v\left(\prod_{k=0}^{n-1}(x-a_k)\right)\,.$$
\begin{proof}
We prove by induction on $n$ that there exists a sequence $a_0,\ldots, a_n$ which is the beginning of a $v$-ordering of $E$. We know that $a_0$ may be any element of $E$. Assume that $a_0,\ldots,a_{n}$ are choosen in such a way that there are the first terms of a $v$-ordering of $E$. We are looking for some $a_{n+1}\in E$ which minimizes $v(h(x))$ where $h(x)=\prod_{k=0}^{n-1}(x-a_k)$ for  $x\in E$. If $\inf_{x\in E}v(h(x))$ was not a minimum, there would exist an infinite sequence $\{x_m\}_{m\geq 0}$ of elements of $E$ such that the sequence $\{v(h(x_m))\}_{m\geq 0}$ is strictly decreasing. Since the sequence $\{v(h(x_m))\}$ is the sum of the $n$ sequences $\{v(x_m-a_i)\}$ of non-negative numbers, at least one of these $n$ sequences admits an infinite subsequence which is strictly decreasing. This is a contradiction with the hypothesis. Thus, there exists $a_{n+1}\in E$ such that $v(h(a_{n+1}))=\inf_{x\in E}v(h(x))$.
\end{proof} 

\begin{lemma}\label{th:33}
If $E$ admits $v$-orderings, then every prime ideal $\PP$ of $\int(E,V)$ lying over the maximal ideal $\mm$ contains the ideal $\int(E,\mm)=\{f\in K[X]\mid f(E)\subseteq \mm\}$.
\end{lemma}

\begin{proof}
Let $\{a_n\}_{n\geq 0}$ be a $v$-ordering of $E$. For each $n\geq 0$, let $f_n(X)=\prod_{k=0}^{n-1}\frac{X-a_k}{a_n-a_k}$. One knows that the polynomials $f_n$ form a basis of the $V$-module $\int(E,V)$ (cf.\cite{bib:bhargava1}). Let $f\in\int(E,\mm)$ and write: $$f(X)=\sum_{k=0}^dc_kf_k(X)\textrm{ with }c_k\in V.$$ Clearly, for instance by means of the recursive formula $$c_k=f(a_k)-\sum_{h=0}^{k-1}c_hf_h(a_k),$$ one has:
$$\inf_{a\in E}v(f(a))=\min_{0\leq k\leq d}v(c_k)=\min_{0\leq k\leq d}v(f(a_k))=\delta>0.$$
Let $t\in \mm$ be such that $v(t)=\delta$, then $\frac{1}{t}f(X)\in\int(E,V)$, and hence, $f(X)\in t\,\int(E,V)\subseteq \PP$.
\end{proof}

Now, we see what can be said when moreover there is no pseudo-stationary sequence.

\begin{lemma}\label{th:33D}
If $E$ does not contain any pseudo-divergent sequence or any pseudo-stationary sequence, then every polynomial $f(X)\in\int(E,V)$	takes on $E$ only finitely many values modulo $\mm$.
\end{lemma}

\begin{proof}
Let $\{a_n\}_{n\geq 0}$ be a $v$-ordering of $E$. Fix some $n\geq 0$ and let us prove first the conclusion for the polynomial $f_n(X)=\prod_{k=0}^{n-1}\frac{X-a_k}{a_n-a_k}$. Let $h_n(X)=\prod_{k=0}^{n-1}(X-a_k)$ and $w_E(n)=v(\prod_{k=0}^{n-1}(a_n-a_k))$. By definition of $v$-orderings, for every $x\in E$, $v(h_n(x))\geq w_E(n)$. If $v(h_n(x))>w_E(n)$, then $f_n(x)\equiv 0\pmod{\mm}$. 	We may restrict our study to the elements $x$ of $E_0=\{x\in E\mid v(h_n(x))=w_E(n)\}$. 

We claim that, for each $0\leq k<n$, the set $N_k=\{v(x-a_k)\mid x\in E_0\}$  is finite. Assume that there exists some $k$ such that the set $N_k$ is infinite. Then, it contains a stricly increasing or strictly decreasing sequence $\{v(x_m-a_k)\}_{m\geq 0}$, in fact, the sequence is stricly increasing, since otherwise the sequence $\{x_m\}_{m\geq 0}$ would be pseudo-divergent, contrarily to the hypothesis. It follows then from the equality $v(h_n(x))=w_E(n)$ that there exists at least another $k'$ such that the set $\{v(x_m-a_{k'})\mid m\geq 0\}$ is also infinite. For the same reason, we can extract from the sequence $\{x_m\}_{m\geq 0}$ a sequence $\{x'_m\}_{m\geq 0}$ such that the sequence $\{v(x'_m-a_{k'})\}_{m\geq 0}$ is strictly increasing. And so on $\ldots$ and we reach a contradiction.

Let $S$ be a set of representatives of $V$ modulo $\mm$ and, for each $\nu\in v(V^*)$, choose some $d_\nu\in V$ such that $v(d_\nu)=\nu$. We claim that, for every $k\in \{0,1,\ldots,n-1\}$ and every $\nu\in N_k$, the set $$S_{k,\nu}=\{s\in S\mid \exists x\in E_0\; v(x-a_k)=\nu \textrm{ and }\frac{x-a_k}{d_\nu}\equiv s\pmod{\mm}\}$$ is finite. Assume that, for some $k$ and some $\nu\in N_k$, this set is infinite. Then, we could extract from $E_0$ a sequence $\{x_m\}_{m\geq 0}$ such that $v(x_m-a_k)=\nu$ and $\frac{x_m-a_k}{d_\nu}\equiv s_m\pmod{\mm}$ where all the $s_m$'s are distinct. Consequently, for all $m\not= l$, we would have $v(x_m-x_l)=\nu$ and the sequence $\{x_m\}_{m\geq 0}$ would be stationary.

Finally, the number of values of $f_n(x)$ modulo $\mm$ for $x\in E$ is finite since it is less or equal to $1+\prod_{k=0}^{n-1} \textrm{Card}(N_k)\times\max\,\{\textrm{Card}(S_{k,\nu})\mid\nu\in N_k\} $. Now, let $f(X)$ be any polynomial of $\int(E,V)$ of degree $d$. The values of $f$ on $E$ are linear combinations of those of the polynomials $f_0, f_1,\ldots, f_d$ with coefficients $c_0,c_1,\ldots,c_d\in V$ which depend only on $f$. Consequently, there are finitely many modulo $\mm$. 
\end{proof}

Recall Frisch's following result:

\begin{lemma}\cite[Lemma 5.1]{bib:frisch2017}\label{th:34D}
If every polynomial of $\int(E,V)$ takes only finitely many values modulo $\mm$, then every prime ideal of $\int(E,V)$ containing $\int(E,\mm)$ is maximal with residue field isomorphic to $V/\mm$.
\end{lemma}

\begin{proof}
Let $\PP$ be a prime ideal of $\int(E,V)$ containing $\int(E,\mm)$. Let $S$ be a set of representatives of $V$ modulo $\mm$. Let $f\in\int(E,V)$ and let $a_1,\cdots,a_r\in S$ be a set of representatives of the residue classes modulo $\mm$ of the values of $f$ on $E$. Then, $\prod_{i=1}^r(f-a_i)$ is in $\int(E,\mm)\subseteq \PP$. Consequently, there exists $i$ such that $f-a_i\in\PP$. Thus, $S$ is also a set of representatives of $\int(E,V)$ modulo $\PP$.
\end{proof}

Now, we are able to describe the spectrum of the ring $\int(E,V)$ when it is a Pr\"ufer domain, and more generally:

\begin{theorem}\label{th:35D}
Assume that $E$ does not contain any pseudo-divergent or pseudo-stationary sequence. Then, the prime ideals of $\int(E,V)$ are the following:
\begin{enumerate}
\item The nonzero prime ideals lying over the ideal $(0)$ of $V$ are in one-to-one correspondence with the monic irreducible polynomials $q$ of $K[X]$. To the polynomial $q$ corresponds the ideal:
$$\PP_q=\{qh\in\int(E,V)\mid h\in K[X]\,\}=qK[X]\cap\int(E,V)\,.$$
\item The prime ideals lying over the maximal ideal $\mm$ of $V$ are maximal with residue field isomorphic to $V/\mm$. They are the ideals of the form
$$\mm_{\Uu}=\{f\in\int(E,V)\mid f^{-1}(\mm)\in\Uu\}$$
where $\Uu$ is any  ultrafilter on $E$.
\end{enumerate}
\noindent Consequently, $\dim(\int(E,V))=2$. 

\end{theorem}

\begin{proof}
Assertion (1) is a general result without any hypothesis~\cite[V.1.2]{bib:CC}. The first part of assertion (2) is a straightforward consequence of Lemmas~\ref{th:31D}, \ref{th:33}, \ref{th:33D} and \ref{th:34D}. The second part is a particular case of \cite[Theorem 5.3]{bib:frisch2017}.
Finally, note that, for each $a\in E$, the ideal
$$\mm_a=\{f\in\int(E,V)\mid f(a)\in\mm\}$$
which corresponds  to the trivial ultrafilter associated to $a$ is a maximal ideal, and the last assertion follows from the obvious fact that $\PP_{X-a}\subseteq \mm_a$. 
\end{proof}

Among the maximal ideals lying over $\mm$, we just mentioned ideals of the form $\mm_a=\{f\in\int(E,V)\mid f(a)\in\mm\}$ where $a$ is any element of $E$. Classically, there are other ideals of this type, those associated to the elements of the completion of $E:$ let $\hV$, $\hm$, and  $\hE$ denote the completions of $V$, $\mm$, and $E$ for the topology associated to the valuation. Thanks to the continuity of the integer-valued polynomials~\cite{bib:CCL}, we have the containment: $$\int(E,V)\subseteq\int(\hE,\hV)\,.$$ 
Therefore,
$$\forall x\in\hE \quad\quad  \mm_x=\{f\in\int(E,V)\mid f(x)\in\hm\}$$ 
is a maximal ideal of $\int(E,V)$. There is another way to characterize this ideal $\mm_x:$ if $\{x_n\}_{n\geq 0}$ is a sequence of elements of $E$ with limit $x$, then $$\mm_x=\{f\in\int(E,V)\mid f(x_n)\in\mm \textrm{ for almost all } n\,\}\,.$$
This is the ideal associated to the filter formed by the cofinite subsets of the set $\{x_n\mid n\geq 0\}$. 
That suggest the idea of other ideals lying over $\mm:$ if $\{x_n\}_{n\geq 0}$ is  an infinite sequence of distinct elements of $E$, we may consider the ideal
$$\mm_{\{x_n\}}=\{f\in\int(E,V)\mid f(x_n)\in\mm \textrm{ for almost all } n\,\}.$$ But, this ideal of $\int(E,V)$ is not necessarily prime: it will be the case if the fact that $f(x_n)\in \mm$ for infinitely many $n$ implies that $f(x_n)\in\mm$ for almost all $n$. Here is such an example:
let $\{x_n\}_{n\geq 0}$ be a pseudo-convergent sequence, that is, a sequence such that the sequence $\{v(x_{n+1})-v(x_n)\}_{n\geq 0}$ is strictly increasing. Clearly, a subsequence of a pseudo-convergent sequence is still pseudo-convergent. As a consequence:

\begin{proposition}
If $\{x_n\}_{n\geq 0}$ is a pseudo-convergent sequence of elements of $E$, then the following ideal is a maximal ideal of $\int(E,V)$ lying over $\mm:$
$$\mm_{\{x_n\}}=\{f\in\int(E,V)\mid f(x_n)\in\mm \textrm{ for almost all } n\,\}.$$	
\end{proposition}

\begin{definition}\cite{bib:ostrowski}
An element $x$ is said to be a  {\em pseudo-limit} of a pseudo-convergent sequence $\{x_n\}_{n\geq 0}$ if the sequence $\{v(x-x_n)\}_{n\geq 0}$ is strictly increasing.
\end{definition}


In fact, we have the following identity between maximal ideals:   

\begin{proposition}\label{th:38u}
If $x$ is a pseudo-limit of a pseudo-convergent sequence $\{x_n\}_{n\geq 0}$, then $\mm_x=\mm_{\{x_n\}}$. 	
\end{proposition}
In order to prove this equality, recall:

\begin{lemma}\cite{bib:ostrowski} or \cite[Lemma 1]{bib:kaplansky}.\label{lem:37}
If $\{x_n\}_{n\geq 0}$ is a pseudo-convergent sequence, then one of the following assertions holds:
\begin{enumerate}
\item $\forall n\quad v(x_{n+1})>v(x_n),$
\item  $\exists n_0 \;\forall n\geq n_0\quad v(x_{n+1})=v(x_n).$
\end{enumerate}
\end{lemma}

\begin{lemma}\cite{bib:ostrowski} or \cite[Lemma 5]{bib:kaplansky}.\label{lem:310}
If $\{x_n\}_{n\geq 0}$ is a pseudo-convergent sequence then, for every $f(X)\in K[X]$, there exists $n_0$ such that the sequence $\{f(x_n)\}_{n\geq n_0}$ is pseudo-convergent.
\end{lemma}

\begin{proof}[Proof of Proposition~\ref{th:38u}]
Since the ideals are maximal, it is enough to prove the containment $\mm_{\{x_n\}}\subseteq \mm_x$.	
Let $f\in\mm_{\{x_n\}}$. There exists $n_1$ such that, for $n\geq n_1$, $f(x_n)$ is pseudo-convergent by Lemma~\ref{lem:310}, $v(f(x_n))\geq v(f(x_{n_1}))$ by Lemma~\ref{lem:37},  and $f(x_n\}\in\mm$ by definition of the ideal $\mm_{\{x_n\}}$. Let $t\in V$ be such that $v(t)=v(f(x_{n_1}))$. Then, $v(t)>0$ and, for $n\geq n_1$, $\frac{1}{t}f(x_n)\in V$. Let $F=\{x_n\}_{n\geq n_1}$, then $\frac{1}{t}f(X)\in\int(F,V)$. Lemma~\ref{lembelow} below implies that $\frac{1}{t}f(x)\in V$, and hence, $f(x)\in tV\subseteq \mm$, that is, $f\in\mm_x$. 
\end{proof}

\begin{lemma}\cite[Theorem 9.2]{bib:fez} or \cite[Proposition 4.8]{bib:chabert2010}.\label{lembelow}
Every pseudo-limit $x$ of a pseudo-convergent sequence contained in $E$ belongs to the polynomial closure of $E$, that is, if $f(E)\subseteq V$, then $f(x)\in V$. 
\end{lemma}

A pseudo-convergent sequence does not always admit a pseudo-limit in $K$, but:

\begin{proposition}\cite[Theorems 2 and 3]{bib:kaplansky} 
A pseudo-convergent sequence in $K$ always admits a pseudo-limit in some immediate extension	of $K$.
\end{proposition}

Recall that an {\em immediate extension} of $K$ is an extension of $L$ of $K$ endowed with a valuation which is an extension of the valuation $v$ of $K$ with the same group of values and the same residuel field. 

\begin{proposition}\cite[Theorem 4]{bib:kaplansky}
A field with a valuation contains a limit for each of its pseudo-convergent sequences if and only if it is maximally complete.
\end{proposition}

Recall that a valuation domain is said {\em maximally complete} if it does not admit any proper immediate extension and that:

\begin{proposition}\cite[Satz 24]{bib:krull}
Every valuation domain $V$ can be embedded in a maximally complete valuation domain $W$ that is an immediate extension of $V$.
\end{proposition}

Thus, let $W$ be a maximally complete valuation domain that is an immediate extension of $V$. Let us still denote by $v$ the extension of $v$ to $W$ and let $\overline{E}^W$ be the polynomial closure of $E$ in $W$. Still assuming that $E$ does not contain any pseudo-divergent or pseudo-stationary sequence, we know that $\overline{E}^W$ is formed by all the pseudo-limits of the pseudo-convergent sequences of $E$ \cite[Theorem 5.2]{bib:chabert2010}. Clearly, 
$$\int(E,V)\subseteq\int(E,W)=\int(\overline{E}^W,W)\,.$$
For each $z\in\overline{E}^W$, we may consider the maximal ideal $\MM_z$ of $\int(E,W)$ defined by $\MM_z=\{f\in \int(E,W)\mid v(f(z))>0\}$ and the corresponding maximal ideal of $\int(E,V):$
$$\mm_z=\{f\in\int(E,V)\mid v(f(z))>0\}\,.$$
We end this section with a conjecture:

\medskip

\noindent{\bf Conjecture}. {\em All the maximal ideals of $\int(E,V)$ lying over $\mm$ are of the form:
$$\mm_z=\{f\in\int(E,V)\mid v(f(z))>0\} \textrm{ where } z\in\overline{E}^W\,.$$
Equivalently, all the maximal ideals of $\int(E,V)$ lying over $\mm$ are of the form:
$$\mm_{\{x_n\}}=\{f\in\int(E,V)\mid f(x_n)\in\mm \textrm{ for almost all } n\}$$
where $\{x_n\}_{n\geq 0}$ is any pseudo-convergent sequence of elements of $E$.}



\NC
 
 \section{$\int(E,V)$ and the stacked bases property}
 
Recall the following classical result:

\begin{proposition}\cite[Theorem 6]{bib:BKU}\label{th:41D}
A Pr\"ufer domain has the stacked bases property if and only if it has the UCS property	
\end{proposition}

To explain what is the UCS property, we need to recall first the notion of content. The {\em content} of a matrix $A$ with coefficients in a ring $R$ is the ideal of $R$ generated by the coefficients of $A$, we denote it by $\cont(A)$. The content of a finitely generated submodule $N$ of a free $R$-module $M$ is the content of the matrix formed by the components of any system of generators of $N$ with respect to any basis of $M$. To say that $N$ has {\em unit content} means that its content is $R$. 

\begin{definition}
A ring $R$ has the {\em unit content summand property} or {\em UCS property} (sometimes called BCS property) if, for every $m$, every finitely generated submodule $N$ of $R^m$ with unit content contains a rank-one projective submodule which is a summand of $R^m$.
 \end{definition}

\begin{theorem}\label{th:43D}
 If $E$ does not contain any pseudo-divergent sequence or any pseudo-stationary sequence, then the ring $\int(E,V)$ has the UCS property. 	
 \end{theorem}
 
 \begin{proof}
Once we know the spectrum of $\int(E,V)$ (Theorem~\ref{th:35D}, in fact Lemma~\ref{th:34D} would be enough), the following proof is very similar to that given by Brewer and Klingler in~\cite{bib:BK2}. The fact that $\int(E,V)$ has the UCS property follows from the fact that $\int(E,V)$ is almost local-global \cite{bib:BK1}, which means that every proper factor ring of $\int(E,V)$ is local-global, which itself means that each polynomial in several variables with coefficients in the ring that represents units locally also represents a unit globally (for definitions and properties see~\cite[Chapter V \S 4]{bib:FS}). It is enough to know that a ring which is either zero-dimensional or semi-local is local-global.

Let $\II$ be a nonzero ideal of $\int(E,V)$ to which we associate two other ideals:
$$\aa=\bigcap_{\PP_q\supseteq \II}\,\PP_q \; \textrm{ and } \; \bb =\mm\,\int(E,V)+\II\,.$$
Note first that $\II$ is contained in at most finitely many ideals $\PP_q$. If $\II$ is not contained in some $\PP_q$, then $\dim(\int(E,V)/\II)=0$ and $\int(E,V)/\II$ is local-global.  If $\II$ is contained in some $\PP_q$, then $\int(E,V)/\aa$ is semi-local, and hence local-global, while $\dim(\int(E,V)/\bb)=0$ and $\int(E,V)/\bb$ is local-global. As $$\{\MM\in \textrm{Max}(\int(E,V))\mid \II\subseteq\MM\}= \{\MM\mid \aa\subseteq\MM\}\bigsqcup \{\MM\mid \bb\subseteq\MM\}\,,$$ we may conclude that $\int(E,V)/\II$ is local-global by applying the following lemma to the ring $R=\int(E,V)/\II$.
 \end{proof}

\begin{lemma}\cite[Lemma 2]{bib:BK2}
Let $R$ be a ring with two nonzero ideals $\Aa$ and $\Bb$ such that $\mathrm{Max}(R)=\mathrm{Max}(R/\Aa)\cup \mathrm{Max}(R/\Bb)$ and $\mathrm{Max}(R/\Aa)\cap\mathrm{Max}(R/\Bb)=\emptyset$. If $R/\Aa$ and $R/\Bb$ are local-global rings, then $R$ is local-global.
\end{lemma}
 
 Lemma~\ref{th:23D} and Proposition~\ref{th:41D} imply then that:
 
 \begin{corollary}
All the rings $\int(E,V)$ which are Pr\"ufer domains have the stacked bases property.	
 \end{corollary}


\section{Global study: the rings $\int(\Oo_K)$}

Let us consider now the classical ring of integer-valued polynomials on $\SZ:$
$$\int(\SZ)=\{f(X)\in\SQ[X]\mid f(\SZ)\subseteq\SZ\}\,$$
and, more generally, for every number field $K$ with ring of integers $\Oo_K$, the two-dimensional Pr\"ufer domain
$$\int(\Oo_K)=\{f(X)\in K[X]\mid f(\Oo_K)\subseteq \Oo_K\}\,.$$
The previous proof given for Theorem~\ref{th:43D} does not work since, for instance, the  ring $\int(\SZ)/(X^2+14)$ is not local-global~\cite[Example 7]{bib:BK2}. We will use the following characterization of the UCS property.

\begin{proposition}\cite{bib:GH}\label{th:50D}
A ring $R$ has the UCS property if and only if, for every matrix $B\in\Mm_{n\times m}(R)$ with unit content, there exists a matrix $C\in\Mm_{m\times l}(R)$ such that the matrix $BC$ has unit content and all its $2\times 2$-minors are zero. 
\end{proposition}

We already studied the case of $\int(\SZ)$ in~\cite{bib:chabert2012} where we proved that in the previous proposition, under some hypotheses, we may assume that $m=l=2$ thanks to the following technical lemma:

\begin{lemma}\cite[Proposition 4.3]{bib:chabert2012}\label{th:54D}
Let $R$ be a ring and $S$ be a multiplicative subset of $R$ without zero divisors such that
\begin{enumerate}
\item the ring $S^{-1}R$ has the BCU property,
\item for every $a\in S$, the ring $R/aR$ has the BCU property.	
\end{enumerate}
Then, for every $n$, every finitely generated submodule of $R^n$ with unit content contains a submodule with unit content which may be generated by two elements. More precisely, we may ask that the contents of these two generators $x$ and $y$ satisfy: $\cont(x)\cap S\not=\emptyset$ and, if $a\in \cont(x)\cap S$, then $\cont(y)+aR=R$.
\end{lemma}

The last assertion about the contents is not in~\cite[Proposition 4.3]{bib:chabert2012} but is in its proof. 
Now, let us recall the definition of the BCU property which is a strong form of the UCS property:

\begin{definition}
A ring $R$ has the {\em BCU property} if, for every matrix $B\in\Mm_{n\times m}(R)$ with unit content there exists a column-matrix $X$ such that the column-matrix $BX$ has unit content.
\end{definition}

Then, one has the following implications (see the proof of \cite[Theorem 1]{bib:BK2}): 

\smallskip

$R$ is local-global $\Rightarrow$ $R$ has the BCU property $\Rightarrow$ $R$ has the UCS property.

\begin{proposition}\label{th:51D}\cite[Corollary 4.4]{bib:chabert2012}
Let $R$ be a ring and $S$ be a multiplicative subset of $R$ without zero divisors such that
\begin{enumerate}
\item the ring $S^{-1}R$ has the BCU property,
\item for every $a\in S$, the ring $R/aR$ has the BCU property.	
\end{enumerate}
Then, $R$ has the UCS property if and only if
$$(*) \left\{\begin{array}{l}\forall B\in\Mm_{n\times 2}(R) \textrm{ such that } \cont(B)=R, \\ \exists C\in \Mm_{2\times 2}(R) \textrm{ such that } \cont(BC)=R \textrm{ and every } 2\times 2 \textrm{-minor of } BC \textrm{ is zero} \end{array}\right.$$
Moreover, we may consider only matrices $B=(B_1 \; B_2)$ such that $\cont(B_1)\cap S\not=\emptyset$ and $\cont(B_2)+aR=R$ where $a\in \cont(B_1)\cap S$. 
\end{proposition}

\smallskip

The last assertion of Proposition~\ref{th:51D} follows from the last one of Lemma~\ref{th:54D}.
In the following proposition we show that, assuming that $R$ is a domain, we may also replace $n$ by 2.

\begin{proposition}\label{th:55E}
Let $R$ be a domain and $S$ be a multiplicative subset of $R$ such that
\begin{enumerate}
\item the ring $S^{-1}R$ has the BCU property,
\item for every $a\in S$, the ring $R/aR$ has the BCU property.	
	
\end{enumerate}
Then, $R$ has the UCS property if and only if
$$(**)\left\{\begin{array}{l} \forall B\in\Mm_{2\times 2}(R) \textrm{ such that } \cont(B)=R,\\ \exists C\in\Mm_{2\times 2}(R) \textrm{ such that } \cont(BC)=R \textrm{ and } \det(BC)=0\,.\end{array}\right. $$
Moreover, we may consider only matrices $B=\left[\begin{array}{cc}a&c\\b&d \end{array}\right]$ such that $$(***) \quad\quad \quad\quad a\in S \textrm{ and } (a,c,d)=R\,.$$
\end{proposition}

\begin{proof}
Let us prove that condition $(**)$ in Proposition~\ref{th:55E} implies condition (*) of Proposition~\ref{th:51D}. 
Let $B=(B_1\;B_2)\in\Mm_{n\times 2}(R)$ such that $\cont(B)=R$ and $\cont(B_1)\cap S\not=\emptyset$. Let $a\in\cont(B_1)\cap S$. If $rk(B)=1$, we may choose $C=I_2$. Thus, we may also assume that rk$(B)=2$. Then, we consider $^tB\in\Mm_{2\times n}(R)$ as the matrix of a submodule $N$ of $R^2$. Since the first row of $^tB$ contains $a$, there exists $x\in N$ whose first component is $a\in S$.  By Lemma~\ref{th:54D}, there exists $y\in N$ such that $\cont(y)+aR=R$. 

$\bullet$ If rk$(\langle x,y\rangle)=2$, let $B_0$ be the matrix of the vectors $x$ and $y$. 

$\bullet$ Otherwise, there exists $z\in N$ such that rk$(\langle x,z\rangle)=2$. Let $B_0$ be the matrix of the vectors $x$ and $y+az$. 

In both cases, $\cont(B_0)=R$, $\det(B_0)\not=0$, and the element $a\in S$ is in the first row and the first column of $B_0$. 
By construction, there exists $D\in\Mm_{n\times 2}(R)$ such that $B_0=\,^tB\cdot D$. By hypothesis $(**)$, there exists $C\in\Mm_{2\times 2}(R)$ such that $\cont(^tB_0\cdot C)=R$ and $\det(^tB_0\cdot C)=0$. In fact, $C\in\Mm_{2\times 2}(R)$ is also suitable for $B\in\Mm_{n\times 2}(R):$

 $\bullet$ $[\,^tD\cdot B\cdot C= \,^tB_0\cdot C$ and $\cont(^tB_0\cdot C)=R\,]\;\Rightarrow\; \cont(B\cdot C)=R.$ 
 
 $\bullet$ [\,$\det(B_0)\not= 0$ and $\det(^tB_0\cdot C)=0\,]$ $\Rightarrow$ $\det(C)=0\,.$
 
 $\bullet$ $\det(C)=0 \;\Rightarrow$ every $2\times 2$-minor of $B\cdot C\in\Mm_{n\times 2}(R)$ is zero.
 
 By the way, we see that the matrix $B_0$ satisfies condition $(***)$ since $c$ and $d$ are the components of either $y$ or $y+az$.
 \end{proof}

Applying Proposition~\ref{th:55E} to the ring of integers in number fields, we obtain:

\begin{theorem}\label{th:56D}
Let $K$ be a number field with ring of integers $\Oo_K$. Then, the two-dimensional Pr\"ufer domain $\int(\Oo_K)=\{f(X)\in K[X]\mid f(\Oo_K)\subseteq \Oo_K\}$ has the stacked bases property if and only if	:
$$\textrm{for every matrix } B\in\Mm_{2\times 2}(\int(\Oo_K)) \textrm{ such that } \cont(B)=\int(\Oo_K)  $$
$$\textrm{there exists } C\in\Mm_{2\times 2}(\int(\Oo_K)) \textrm{ such that } \cont(BC)=\int(\Oo_K) \textrm{ and } \det(BC)=0\,. $$
Moreover we may assume that
$$B=\left[\begin{array}{cc}a&c\\b&d \end{array}\right]\textrm{ with }a\in\Oo_K\setminus\{0,\Oo_K^\times\}\,, \,(a,c,d)=\int(\Oo_K) \textrm{ and } \det(B)\notin \Oo_K $$
\end{theorem}

\begin{proof}
If $R=\int(\Oo_K)$ and $S=\Oo_K\setminus\{0\}$, then

(1) $S^{-1}R=K[X]$ is a principal ideal domain, and hence, has the BCU property.

(2) for every $a\in \Oo_K\setminus\{0\}$, $\dim(\int(\Oo_K)/a\,\int(\Oo_K))=0$ \cite[Proposition V.2.3]{bib:CC}. Consequently, 	$\int(\Oo_K)/a\,\int(\Oo_K)$ is local-global, and hence, has the BCU property.

Finally, with respect to the conditions about the coefficients of $B$, note first that, thanks to the end of Proposition~\ref{th:55E}, we may assume that $a\in \Oo_K$ and that $(a,c,d)=\int(\Oo_K)$.  Secondly, if $a\in\Oo_K^\times$, then the matrix $C=\left(\begin{array}{cc}1&1\\0&0\end{array}\right)$ is suitable and, if $\det(B)\in\Oo_K$, by an argument similar  to that used for the proof of Theorem \cite[V.4.7]{bib:FS}, we can find an $r\in\int(\Oo_K)$ such that $(a+rc,b+rd)=\int(\Oo_K)$, and hence, the matrix $C=\left(\begin{array}{cc}1&1\\r&r\end{array}\right)$ is suitable.
\end{proof}

By restricting the size of the matrices $B$ and $C$ and by adding conditions on the coefficients of $B$, Theorem~\ref{th:56D} may be useful to prove the fact that the ring $\int(\Oo_K)$ where $K$ is a number field has the stacked bases property. On the other hand, if we want to obtain a counterexample, that is, to prove the existence of a two-dimensional Pr\"ufer domain that does not have this property, we can try to strengthen the conclusion as in Proposition~\ref{th:57} below.

\begin{proposition}\label{th:57}
Let $R$ be a ring and let $B=\left(\begin{array}{cc}a&c\\b&d\end{array}\right)\in\Mm_{2\times 2}(R)$ be such that $\cont(B)=R$ and $\det(B)\not=0$. Then, the following assertions are equivalent:
\begin{enumerate}
\item there exists $C\in\Mm_{2\times 2}(R)$ s.t. $\det(C)=0$ and $\cont(BC)=R$.
\item there exists $C\in\Mm_{2\times 2}(R)$ s.t. $\det(C)=0$ and  $\mathrm{Tr}(BC)=1.$
\item there exists $C\in\Mm_{2\times 2}(R)$ s.t. $BC$ is a nontrivial idempotent matrix.
\end{enumerate}
\end{proposition}

\begin{proof}
Obviously (2) implies (1). Assume that (1) holds and let $C=\left(\begin{array}{cc}e&f\\g&h\end{array}\right)$ be such that $\det(C)=0$ and $\cont(BC)=R$. Since $BC=\left(\begin{array}{cc}ae+cf&ag+ch\\be+df&bg+dh\end{array}\right)$, the last condition means that there exist $r,s,t,u\in R$ such that $$r(ae+cf)+s(be+df)+t(ag+ch)+u(bg+dh)=1,$$
that is,
$$a\alpha+b\beta+c\gamma+d\delta=1$$ 
where
$$\alpha=re+tg\;;\;\beta=se+ug\;;\;\gamma=rf+th\;;\;\delta=sf+uh.$$
Note that, if $D=\left(\begin{array}{cc}r&s\\t&u\end{array}\right)$, then $C_0=CD=\left(\begin{array}{cc}\alpha&\beta\\ \gamma&\delta\end{array}\right)$. Consequently,
$$\alpha\delta-\beta\gamma=\det(C_0)=0.$$
Moreover, 
$$BC_0=\left(\begin{array}{cc}a&c\\b&d\end{array}\right) \left(\begin{array}{cc}\alpha&\beta\\ \gamma&\delta\end{array}\right)=\left(\begin{array}{cc}a\alpha+c\gamma&a\beta+c\delta\\ b\alpha+d\gamma&b\beta+d\delta\end{array}\right)$$ shows  that $\mathrm{Tr}(BC_0)=1$. Thus, $C_0$ satisfies assertion (2).

Assume now that $C$ satisfies $\det(C)=0$ and $\mathrm{Tr}(BC)=1$. The characteristic polynomial of $BC$ is then $X^2-X$, and hence, $BC$ is idempotent. Moreover, $BC\not=0$ since $\mathrm{Tr}(BC)=1$ and $BC\not=I_2$ since $\det(BC)=0$. Conversely, if $BC$ is a nontrivial idempotent matrix, the minimal polynomial of $BC$ divides $X^2-X$ and is distinct from $X$ and $X-1$. This minimal polynomial is then  $X^2-X$, this is also the characteristic polynomial of $BC$. Consequently, $\mathrm{Tr}(BC)=1$.
\end{proof}

\begin{remark}
Assertion (2) of Proposition~\ref{th:57} may be formulated in the following way: if the elements $a,b,c,d$ of $R$ are coprime and if $ad\not=bc$, then there exists a `strong Bezout relation' between $a,b,c,d,$ that is, elements $\alpha,\beta,\gamma,\delta$ in $R$ such that:
$$a\alpha+b\beta+c\gamma+d\delta=1 \textrm{ and }\alpha\delta=\beta\gamma.$$
\end{remark}

\begin{corollary}
If the integers $a,b,c,d$ are coprime, then there exist integers $\alpha,\beta,\gamma,\delta$ such that:
$$a\alpha+b\beta+c\gamma+d\delta=1 \textrm{ and }\alpha\delta=\beta\gamma.$$
\end{corollary}

\begin{proof}
This is a consequence of the fact that $\SZ$ is principal. Let us give a direct proof without assuming that $ad\not= bc$. Let $M$ be the submodule of $\SZ^2$ generated by the vectors $\binom{a}{b}$ and $\binom{c}{d}$. It follows from the simultaneous bases property of the principal ideal domains that there exist a basis $(e_1,e_2)$ of $\SZ^2$ and integers $\alpha_1$ and $\alpha_2$ such that $\alpha_1\vert\alpha_2$ and  $(\alpha_1 e_1,\alpha_2 e_2)$ is a basis of $M$. As $\cont(M)=\SZ$, we have $\alpha_1=\pm 1$ and $e_1\in M$. Consequently, there exist $\lambda$ and $\mu\in\SZ$ such that $e_1=\lambda\binom{a}{b}+\mu\binom{c}{d}$. As $\cont(e_1)=\SZ$, there exist $u$ and $v\in\SZ$ such that $u(\lambda a+\mu b)+v(\lambda c+\mu d)=1$. Letting $\alpha=u\lambda$, $\beta=u\mu$, $\gamma=v\lambda$ and $\delta=v\mu$, we have $a\alpha+b\beta+c\gamma+d\delta=1$ and $\alpha\delta=\beta\gamma$.
\end{proof}

\begin{example}
Looking for a counterexample for $\int(\SZ)$, we consider the matrix $B=\left(\begin{array}{cc}2&X\\X+1&3\end{array}\right)$. Do there exist $\alpha,\beta, \gamma,\delta\in\int(\SZ)$ such that
$$(\circ)\quad 2\alpha+(X+1)\beta+X\gamma+3\delta=1\;\;\;\;\textrm{and}\;\;\;\;
(\circ\circ)\quad \alpha\delta=\beta\gamma\,.$$
As $2$ and $3$ are coprime, all the 4-tuples $\alpha,\beta, \gamma,\delta\in \intz$ satisfying $(\circ)$ may be obtained in the following way: choose arbitrarily $\beta$ and $\gamma$ in $\intz$, let $$f(X)=(X+1)\beta(X)+X\gamma(X)-1\in\intz,$$ 
then equality $(\circ)$ is equivalent to the following:
$$ 2(\alpha-f)+3(\delta+f)=0\,.$$
Clearly, $\frac{3}{2}(\delta+f)=f-\alpha\in\intz$ implies $\frac{1}{2}(\delta+f)=u\in\intz$. Thus, $(\circ)$ is equivalent to:
$$\alpha=3u+f\;\;\textrm{ and }\;\; \delta=-2u-f \quad \mathrm{where} \;\beta,\gamma,u\in\intz\,.$$ 
Is there a solution satisfying $(\circ\circ) ?$ Equality $(\circ\circ)$ is equivalent to:
$$6u^2+5fu+f^2+\beta\gamma=0$$
or 
$$(12u+5f)^2=f^2-24\beta\gamma\,.$$
Thus, necessarily, $\beta$ and $\gamma$ have to be chosen such that $f^2-24\beta\gamma=g^2$ where $g\in\intz$.  Moreover, the polynomial $g$ has to be such that $\frac{1}{12}(g-5f)\in \intz$. 
Surprisingly (and unfortunately), these two constraints are achievable with\footnote{\,This solution was kindly communicated to us by David Adam.}:
$$\left\{\begin{array}{l}\beta=-X^5-5X^4-2X^3+20X^2+31X+13\\
\gamma=X^3+4X^2+4X\\
g=-X^6-6X^5-6X^4+22X¨3+43X^2+8X-12\,.\end{array}\right.$$	

\end{example}

\end{document}